%% file: abelian.tex
\documentclass[11pt]{amsart}
\usepackage{amsmath,amsfonts,mathrsfs}

\newif\ifdraft
\draftfalse
\input{macros.tex}

\def\Pic{{\rm Pic}}

\newtheorem*{thmA'}{Theorem~A$^\prime$}

\begin{document}

\title{Kodaira dimension of fibrations over abelian varieties}

\author[F.~Meng]{Fanjun~Meng}
\address{Department of Mathematics, Johns Hopkins University, 
3400 N. Charles Street, Baltimore, MD 21218, USA} 
\email{{\tt fmeng3@jhu.edu}}

\author[M.~Popa]{Mihnea~Popa}
\address{Department of Mathematics, Harvard University, 
1 Oxford Street, Cambridge, MA 02138, USA} 
\email{{\tt mpopa@math.harvard.edu}}


\subjclass[2010]{}

\keywords{Kodaira dimension; abelian varieties; positivity}

\thanks{MP was partially supported by the NSF grant DMS-2040378.}

\begin{abstract}
We prove an additivity result for the log Kodaira dimension of algebraic fiber spaces over abelian varieties, a superadditivity result for fiber spaces over varieties of maximal Albanese dimension, as well as a subadditivity result for log pairs.
\end{abstract}

\maketitle

\subsection{Introduction}
In this paper we show the validity,  in the case of algebraic fiber spaces over abelian varieties, of some conjectures on the behavior of Kodaira dimension proposed by the second author; see \cite[Conjectures 3.1 and 4.1]{Popa}. We also deduce 
(a stronger version of) Conjecture 3.1 in \emph{loc. cit.} in the more general case of fiber spaces over varieties of maximal Albanese dimension. We always work over the complex numbers.

\begin{intro-theorem}\label{main}
Let $f \colon X \to A$ be an algebraic fiber space, with $X$ a smooth projective variety and $A$ an abelian variety. Assume that $f$ is smooth over an open set $V \subseteq A$, and denote $U = f^{-1} (V)$ and the general fiber of $f$ by $F$. Then 
\begin{enumerate}
\item $\kappa (V) + \kappa (F) = \kappa (U) \ge \kappa (X)$.
\medskip
\item If $A \smallsetminus V$ has codimension at least $2$ in $A$ (e.g. if $f$ is smooth), then $\kappa (U) = \kappa (X) = \kappa (F)$. Moreover, there exists an \'etale cover $X'$ of $X$ such that
$$P_m (X') = P_m (F) \ge P_m (X) \,\,\,\,\,\,{\rm for ~all}\,\,\,\, m \ge 0.$$
\end{enumerate}
\end{intro-theorem}

Here, for a smooth quasi-projective variety $V$, $\kappa(V)$ denotes the log Kodaira dimension, defined as follows: for any smooth projective compactification $Y$ of $V$ such that $D = Y \smallsetminus V$ is a divisor with simple normal crossings, we have $\kappa (V) = \kappa (Y, K_Y + D)$. Also, for a smooth projective variety $X$ and $m\ge 0$, $P_m (X) = h^0 (X, \omega_X^{\otimes m})$ is the $m$-th plurigenus of $X$.

Item (2) recovers (and strengthens) in particular \cite[Corollary 3.1]{PS2}, stating that if $f$ is smooth, then 
$\dim X - \kappa (X) \ge \dim A$. As for the proof, thanks to the structure of effective divisors on abelian varieties, item (1) is in fact a consequence of (2). On the other hand, Iitaka's $C_{n,m}$ conjecture on the subadditivity of the Kodaira dimension is known for algebraic fiber spaces over abelian varieties by \cite{CP} (see also \cite{HPS}), giving $\kappa (X) \ge \kappa (F)$ for any such. Thus, up to some technicalities regarding the log Kodaira dimension, the key point is the last statement in (2), for which we employ techniques from \cite{LPS} regarding the Chen-Jiang decomposition, as well as a hyperbolicity-type result from \cite{PS3}. We prove in fact a stronger statement:

\begin{intro-theorem}\label{trivial}
Let $f \colon X \to A$ be a surjective morphism from a smooth projective variety to an abelian variety.
Assume that $f$ is smooth away from a closed set of codimension at least $2$ in $A$, and denote its general fiber by $F$.
Then for every $m\ge 1$ we have
$$f_* \omega_X^{\otimes m} \simeq \bigoplus_{i =1}^{P_m (F)} \alpha_i,$$
where $\alpha_i \in {\rm Pic}^0 (A)$ are (possibly repeated) torsion line bundles. In particular, if $f_* \omega_X^{\otimes m}$
is globally generated for some $m$, then 
$$f_* \omega_X^{\otimes m} \simeq \shO_A^{\oplus P_m(F)}.$$
\end{intro-theorem}

We remark that non-trivial torsion line bundles can indeed appear in the decomposition of $f_* \omega_X^{\otimes m}$ 
in Theorem \ref{trivial}. For instance, if $X$ is a bielliptic surface obtained as a quotient of a product of elliptic curves by a non-trivial finite group, then its Albanese map $f$ is smooth, while $P_1(X) = 0$ and $P_1 (F) = 1$; thus $f_* \omega_X$ is a non-trivial torsion line bundle.

Note that assuming Viehweg's $C_{n,m}^+$ conjecture, which holds when $F$ has a good minimal model by \cite{Kawamata2}, the result in Theorem \ref{main}(2) implies that a morphism $f \colon X \to A$ which is smooth in codimension $1$ satisfies ${\rm Var}(f) = 0$, i.e. $f$ is birationally isotrivial (cf. also \cite[Corollary 3.2]{PS2}). By analogy with Ueno's Conjecture $K$, we propose the following strengthening:

\begin{intro-conjecture}\label{conj:splitting}
Let $f \colon X \to A$ be an algebraic fiber space, with $X$ a smooth projective variety, $A$ an abelian variety, and general fiber $F$ with $\kappa (F) \ge 0$. If $f$ is smooth away from a closed set of codimension at least $2$ in $A$ , then there exists an isogeny  $A' \to A$ such that 
$$X\times_{A} A' \sim F\times A',$$
i.e. $X$ becomes birational to a product after an \'etale base change.
\end{intro-conjecture}

When $f$ is smooth with canonically polarized fibers, i.e. $K_F$ is ample, this is proved, with isomorphism, in \cite[\S2]{Kovacs}. A consequence of Theorem \ref{trivial} is that this in fact always holds at the level of canonical models; here $X_{{\rm can}, f}$ denotes the relative canonical model of $f$ (birational to $X$ over $A$ when $F$ is of general type), and $F_{\rm can}$ the canonical model of $F$.

\begin{intro-corollary}\label{splitting}
Under the assumptions of Conjecture \ref{conj:splitting}, there exists an isogeny  $A' \to A$ such that 
$$X_{{\rm can}, f}\times_{A} A' \simeq F_{{\rm can}} \times A'.$$
In particular Conjecture \ref{conj:splitting} holds when $F$ is of general type.
\end{intro-corollary}

\bigskip

It turns out that much of the result in Theorem \ref{main} holds more generally over any variety of maximal Albanese dimension, i.e. endowed with a (not necessarily surjective) generically finite morphism to an abelian variety. We state this separately, since the proof is a rather involved reduction to the case of abelian varieties, combined with Theorem \ref{main}.

\begin{intro-theorem}\label{mad}
Let $f \colon X \to Y$ be an algebraic fiber space, with $X$ and $Y$ smooth and projective, and $Y$ of maximal Albanese dimension. Assume that $f$ is smooth over an open set $V \subseteq Y$, and denote $U = f^{-1} (V)$ and the general fiber of $f$ by $F$. Then 
\begin{enumerate}
\item $\kappa (V) + \kappa (F) \ge \kappa (U) \ge \kappa (X)$.
\item If $Y \smallsetminus V$ has codimension at least $2$ in $Y$, then 
$$\kappa (X) = \kappa (U) = \kappa (V) + \kappa (F) = \kappa (Y) + \kappa (F).$$
\end{enumerate}
\end{intro-theorem}

As predicted in \cite[Conjecture 4.1]{Popa}, it should in fact always be the case that $\kappa (U) = \kappa (V) + \kappa (F)$; what is left is the standard conjectural subadditivity of the log Kodaira dimension, i.e. $\kappa (U) \ge \kappa (V) + \kappa (F)$. For this, unlike in the case of abelian varieties, further ideas related to log canonical pairs seem necessary.  Note that the usual subadditivity $\kappa (X) \ge \kappa (F) + \kappa (Y)$ is known to hold for any algebraic fiber space over $Y$ of maximal Albanese dimension, by \cite[Theorem 1.1]{HPS}.

In a different but related direction,  we remark that using the klt version provided in \cite{Meng} (see also \cite{Jiang} for related results) of the global generation result for direct images in \cite{LPS} used for the results above, one can go beyond the results of \cite{CP} and \cite{HPS}. Concretely, we show the validity of the \emph{log} Iitaka conjecture over an abelian variety $A$, in the case of divisors that do not dominate $A$ (though the statement can be phrased slightly more generally). For a divisor $D$ on a smooth projective variety $X$, we use the notation $P_m (X, D) : = h^0 (X, \omega_X(D)^{\otimes m})$.

\begin{intro-theorem}\label{subadditivity}
Let $f\colon X \to A$ be an algebraic fiber space, with $X$ a smooth projective variety, and $A$ an abelian variety. Denote by $F$ the 
general fiber of $f$. Let $D$ be a reduced effective divisor on $A$ and $E$ a reduced effective divisor on $X$ such that 
${\rm Supp}(f^*D)\subseteq E$. Then
$$\kappa (X,K_X + E) \ge \kappa (F) + \kappa (A,D).$$
More precisely, there exists an \'etale cover $\psi \colon X' \to X$ and a fixed integer $k \ge 1$ such that
$$P_{mk} (X', \psi^* E) \ge P_{mk} (F) \cdot P_m (A, D) \,\,\,\,\,\,{\rm for ~all}\,\,\,\, m \ge 0.$$
\end{intro-theorem}

Note that unlike in the previous results, we are not assuming $f$ to be smooth over $A \smallsetminus D$.
In this setting, the most general case of the log Iitaka conjecture allows for divisors $E$ that dominate $A$, in other words replacing $\kappa(F)$ by $\kappa (F, K_F + E|_{F})$. It would be solved with 
similar methods if extensions of the results in \cite{LPS} were available in the setting of log canonical pairs; this is however 
beyond our reach at the moment. 

One final word: in the results of this paper, morphisms that are smooth away from a closed subset of codimension at least $2$ behave just like morphisms that are smooth everywhere. However, working under this weaker hypothesis leads to 
additional technical arguments.  In order to isolate the main ideas, it may be helpful to assume global smoothness at a first reading.

\subsection{Background}
We will make use of the following two results from \cite{LPS}, which are shown in \emph{loc. cit.} to be equivalent.

\begin{theorem}[{\cite[Theorem B]{LPS}}]\label{thm:LPS1}
Let $f \colon X \to A$ be a morphism from a smooth projective variety to an abelian
variety. Then there exists an isogeny $\varphi \colon A' \to A$ such that if 
\[
\begin{tikzcd}
X' \dar{f'} \rar & X \dar{f} \\
A' \rar{\varphi} & A
\end{tikzcd}
\]
is the fiber product, then  $\varphi^* f_* \omega_X^{\otimes m} \simeq f'_* \omega_{X'}^{\otimes m}$ is globally generated for every $m \geq 1$.
\end{theorem}

Recall that an (algebraic) fiber space is a surjective projective morphism with connected fibers.
We remark, as could have been done already in \cite{LPS}, that Theorem \ref{thm:LPS1} implies a strengthening of the subadditivity of the Kodaira dimension for algebraic fiber spaces over abelian varieties proved in \cite{CP} and later also in \cite{HPS}.

\begin{corollary}\label{old}
Let $f \colon X \to A$ be an algebraic fiber space over an abelian variety, with $X$ smooth and projective, and with general fiber $F$. Then there exists an \'etale cover $X' \to X$, base changed from an \'etale cover of $A$, such that 
$P_m (X') \ge P_m (F)$ for all $m \ge 1$.
\end{corollary}
\begin{proof}
We consider the construction in Theorem \ref{thm:LPS1}, so that $f'_* \omega_{X'}^{\otimes m}$ is globally generated for every $m \geq 1$. The number of independent global sections of this sheaf 
 is therefore at least equal to its rank, which is equivalent to the inequality
$$P_m (X') \ge P_m (F).$$ 
\end{proof}

For the next statement, and for later use in the paper, recall that for a coherent sheaf $\shF$ on an abelian variety $A$, and for $i \ge 0$, we consider the  $i$-th cohomological support locus of $\shF$ given  by
$$V^i (A, \shF) = \{\alpha \in \Pic^0(A) ~|~  H^i(A, \shF \tensor \alpha) \neq 0 \}.$$
We will use the following standard terminology: 
\begin{itemize}
\item $\shF$ is a \emph{GV-sheaf} if $\codim_{\Pic^0 (A)} V^i (A, \shF) \ge i$ for all $i > 0$.
\item $\shF$ is an \emph{M-regular sheaf} if $\codim_{\Pic^0 (A)} V^i (A, \shF) > i$ for all $i > 0$.
\end{itemize}

\begin{theorem}[{\cite[Theorem C]{LPS}}]\label{thm:LPS2}
In the setting of \theoremref{thm:LPS1}, there exists a finite direct sum decomposition
\[
	f_* \omega_X^{\otimes m} \simeq \bigoplus_{i \in I} \bigl( \alpha_i \tensor \pu_i \shF_i \bigr),
\]
where each $p_i \colon A \to A_i$ is a quotient morphism with connected fibers to an
abelian variety, each $\shF_i$ is an M-regular coherent sheaf on $A_i$, and each
$\alpha_i \in \Pic^0(A)$ is a line bundle that becomes trivial when pulled back by
the isogeny in \theoremref{thm:LPS1}.
\end{theorem}

The direct sum decomposition in this last theorem goes under the name of a \emph{Chen-Jiang decomposition}.
The only special fact about $M$-regular sheaves that we will use here is that they are ample, as shown in \cite{Debarre}; similarly, $GV$-sheaves are nef, as shown in \cite{PP2}.

\smallskip

We will also need a hyperbolicity-type result, proved in this generality in \cite{PS3}; it relies on important
ideas and results of Viehweg-Zuo and Campana-P\u aun described in \emph{loc. cit.}, as well as on the theory of Hodge modules. For simplicity, in the following statement we combine two results, and only state the special consequence needed in this paper.\footnote{The statement works in a more general setting; it is comparatively much simpler however, avoiding semistable reduction tricks, when the base is a priori known not to be uniruled.}

\begin{theorem}[{\cite[Theorem 4.1 and Theorem 3.5]{PS3}}]\label{thm:PS3}
Let $f \colon X \to Y$ be an algebraic fiber space between smooth projective varieties, such that $Y$ is not uniruled.
Assume that $f$ is smooth over the complement of a closed subset $Z \subset Y$, and that there exists $m\ge 1$ such that ${\rm det} f_* \omega_{X/Y}^{\otimes m}$ is big. Denote by $D$ the union of the divisorial components of $Z$.  Then the pair $(Y,D)$ is of log general type, i.e. the line bundle $\omega_Y (D)$ is big.
\end{theorem}

\begin{remark}
The theorem above is stated in \emph{loc. cit.} only when $Z = D$, but the proof shows more generally the statement above, as all the objects it involves can be constructed from $Y$ with any closed subset of codimension at least $2$ removed.
\end{remark}

Finally, we will sometimes make use of a more flexible interpretation of the log Kodaira dimension on ambient varieties 
of non-negative Kodaira dimension.

\begin{lemma}\label{lkd}
Let $X$ be a smooth projective variety with $\kappa(X)\ge 0$ , $Z\subseteq X$ a closed reduced subscheme, and $V = X \smallsetminus Z$. Assume that $Z=W\cup D$ where $\codim_X W\ge 2$ and $D$ is a divisor. Then
$$\kappa(V)=\kappa(X , K_X +D).$$
\end{lemma}
\begin{proof}
Take a resolution $\mu\colon Y\to X$  of $(X, Z)$ such that $\mu$ is an isomorphism over $X\smallsetminus Z$; in particular 
$\mu^{-1}(Z)$ is a divisor with simple normal crossing support $G$. Hence by definition
$$\kappa(V )=\kappa(Y, K_Y+ G).$$
Since $\kappa(X)\ge 0$, we deduce that $K_X\sim_{\mathbb{Q}} E$, where $E$ is an effective $\mathbb{Q}$-divisor on $X$. We therefore have that 
$$K_Y\sim \mu^*K_X+F\sim_{\mathbb{Q}} \mu^*E+F,$$
 where $F$ is an effective and $\mu$-exceptional divisor on $Y$ such that $\Supp(F)={\rm Exc}(\mu)$. Now
$$\kappa(X, K_X+D)=\kappa(Y, \mu^*(K_X+D)+F)=\kappa(Y, \mu^*E+\mu^*D+F)$$
while
$$\kappa(Y, K_Y+ G)=\kappa(Y, \mu^*E+F+ G).$$
Since $\Supp(\mu^*E+\mu^*D+F)=\Supp(\mu^*E+F+G)$, we obtain the conclusion.
\end{proof}

\subsection{Proof of Theorem \ref{trivial}}
Let $f\colon X \to A$ be a surjective morphism of smooth projective varieties, with $A$ an abelian variety, and let $F$ be 
the general fiber of $f$.  Assume that $f$ is smooth over an open subset $V\subseteq A$ whose complement has codimension at least $2$. We divide the proof of Theorem \ref{trivial} into two steps:

\noindent
\emph{Step 1.} 
We first prove the last assertion; namely if $f_* \omega_X^{\otimes m}$
is globally generated for some $m\ge 1$, then 
\begin{equation}\label{eq1}
f_* \omega_X^{\otimes m} \simeq \shO_A^{\oplus P_m(F)}.
\end{equation}
We first prove this result when $f$ is a fiber space.  We use the statement and notation of Theorem \ref{thm:LPS2}, and 
we claim in fact that $\dim A_i=0$ for all $i\in I$ appearing in the decomposition in that theorem.
Since $f_*\omega_X^{\otimes m}$ is globally generated and the rank of $f_* \omega_X^{\otimes m}$ is $P_m(F)$, this immediately implies ($\ref{eq1}$).

Assume on the contrary that we have $\dim A_k>0$ for one of the quotients $p_k \colon A \to A_k$. (This includes the case when $A_k = A$.) Denote the kernel of $p_k$ by $C$; this is an abelian subvariety of $A$. By Poincar\'e's complete reducibility theorem, there exists an abelian variety $B\subseteq A$ such that $B+C=A$ and $B\cap C$ is finite, so that the natural morphism 
$\varphi\colon B\times C\to A$ is an isogeny. We consider the following commutative diagram, $q$ is the projection onto $C$, $c\in C$ is a general point, and $f'$ and $f'_{c}$ are obtained by base change from $f$ via $\varphi$ and the inclusion 
$i_c$ of the fiber $B_c$ of $q$ over $c$ respectively:
\[
\begin{tikzcd}
X'_c\rar \dar{f'_c} & X' \rar{\varphi'} \dar{f'} \rar & X \dar{f}  \\
B_c \dar \rar{i_c} & B\times C \dar{q} \rar{\varphi} & A \dar{p_k} \\
\{c\}\rar& C  & A_k
\end{tikzcd}
\]
Note that by construction the composition
$$\psi_c := p_k\circ\varphi\circ i_c \colon B_c \to A_k$$ 
is an isogeny. Furthermore, $X'$ is smooth, since $\varphi$ is \'etale. We denote by $Z \subset A$ the closed subset of 
codimension $\ge 2$ such that $f$ is smooth over $A \smallsetminus Z$. If $Z' := \varphi^{-1} (Z)$, then $Z'$ has codimension $\ge 2$ in $B \times C$ as well, and $f'$ is smooth over its complement. Moreover, $c$ can be chosen sufficiently general so that $X'_{c}$ is smooth (by generic smoothness applied to $q \circ f'$) and $\codim_{B_c}  i_c^{-1}(Z')\ge 2$ as well, hence $f'_c$ inherits the same property as $f$.

\smallskip
\noindent
{\bf Claim.} The line bundle $\det \big( (f'_c)_*\omega_{X'_c}^{\otimes m}\big)$ is big (hence ample, as $B_c$ is an abelian variety).

Assuming the Claim, we can immediately conclude the proof. Indeed, by Theorem \ref{thm:PS3} this would imply
that $\omega_{B_c}$ is big, which is a contradiction and shows that we cannot have $\dim A_k > 0$.

We are left with proving the Claim. To this end, applying the base change theorem as in \cite[Proposition 4.1]{LPS}, since $c$ is general we have
\begin{equation}\label{eq4}
(f'_c)_*\omega_{X'_c}^{\otimes m} \simeq (f'_c)_*(\omega_{X'}^{\otimes m}|_{X'_c})\simeq i_c^*(f'_*\omega_{X'}^{\otimes m})\simeq i_c^*\varphi^* ( f_* \omega_X^{\otimes m}).
\end{equation}
To analyze this, we need to look more carefully at the decomposition of $f_*\omega_X^{\otimes m}$ from Theorem \ref{thm:LPS2}.
Since $f_*\omega_X^{\otimes m}$ is globally generated, we deduce that $\alpha_k\otimes p_k^*\shF_k$ is globally generated, and in particular $h^0(A, \alpha_k\otimes p_k^*\shF_k)>0$. It follows that $\alpha_k$ is trivial on the general, hence every, fiber of $p_k$, and so 
there exists a torsion line bundle $\beta_k\in\Pic^0(A_k)$ such that $\alpha_k\simeq p_k^*\beta_k$. Moreover, since $p_k$ has connected fibers,
$${p_k}_*(\alpha_k\otimes p_k^*\shF_k)\simeq \beta_k\otimes\shF_k$$
by the projection formula (applied to the $0$-th cohomology of $\derR {p_k}_*(\alpha_k\otimes p_k^*\shF_k)$).
If we denote $\shG : = \bigoplus_{i\neq k}(\alpha_i\otimes p_i^*\shF_i)$, we then have 
\begin{equation}\label{eq3}
f_*\omega_X^{\otimes m}\simeq p_k^* (\beta_k\otimes\shF_k) \oplus \shG.
\end{equation}
These summands have various positivity properties. Since $\beta_k\otimes\shF_k$ is an $M$-regular sheaf on $A_k$, it is ample. On the other hand,  since $f_*\omega_X^{\otimes m}$ is a $GV$-sheaf by \cite[Theorem 1.10]{PS1}, it is nef, hence so is $\shG$. (For all this, see the comments after Theorem \ref{thm:LPS2}.) Since a priori they might not be locally free, it is also useful to record that $\shG$ is a weakly positive sheaf, since $f_*\omega_X^{\otimes m}$ is so 
by \cite[Theorem III]{Viehweg}.

Using ($\ref{eq4}$) and ($\ref{eq3}$), we deduce that
$$(f'_c)_*\omega_{X'_c}^{\otimes m}\simeq i_c^*\varphi^* ( f_* \omega_X^{\otimes m})\simeq \psi_c^*(\beta_k\otimes\mathcal{F}_k) \oplus i_c^* \varphi^*\shG.$$
Note in passing that if $f$ is smooth, and so also $f'$ and $f'_c$ are smooth to begin with, then all of the sheaves above are locally free (by the deformation invariance of plurigenera), so by the previous paragraph this is a sum of an ample and a nef vector bundle, hence its determinant is clearly ample. In general we have to be slightly more careful in order to draw the same conclusion. Since $\psi_c$ is an isogeny, we deduce that $\psi_c^*(\beta_k\otimes\mathcal{F}_k)$ is an ample sheaf as well, while $i_c^* \varphi^*\shG$ is weakly positive just as $\shG$ (being a summand of $(f'_c)_*\omega_{X'_c}^{\otimes m}$). In other words, we have that 
$$(f'_c)_*\omega_{X'_c}^{\otimes m}\simeq \shH_1 \oplus \shH_2,$$ 
with $\shH_1$ ample and $\shH_2$ weakly positive. But an ample sheaf is big in the sense of Viehweg 
(see e.g. \cite[\S2 a)]{Debarre} and \cite[\S5 p.293]{Mori}), and so its determinant\footnote{Recall that the determinant of a torsion-free sheaf $\shF$ of generic rank $r$ is the line bundle $(\wedge^r \shF)^{\vee \vee}$, i.e. the  unique extension of the determinant line bundle from the big open set on which $\shF$ is locally free.} $\det \shH_1$ is a big line bundle by \cite[Lemma 3.2(iii)]{Viehweg2} (see also \cite[5.1.1]{Mori}). On the other hand, 
$\det \shH_2$ is weakly positive, e.g. also by \cite[Lemma 3.2(iii)]{Viehweg2};  for a line bundle  this is the same as being pseudoeffective. Their tensor product is therefore also big, hence finally the line bundle $\det \big( (f'_c)_*\omega_{X'_c}^{\otimes m}\big)$ is big. This concludes the proof in the case of fiber spaces.

If $f$ is not a fiber space, we consider its Stein factorization $f = g\circ h$.
Here $B$ is a normal projective variety, $h\colon X\to B$ is a fiber space, and $g\colon B\to A$ is a finite surjective morphism. Note that $h$ smooth over $B\smallsetminus g^{-1}(Z)$ and $g$  is \'etale over $A\smallsetminus Z$; see e.g. 
\cite[Lemma 2.4]{FG}. By the purity of the branch locus, it follows that $g$ is actually  \'etale over $A$, and thus $B$ is also an abelian variety. Moreover the canonical morphism
$$g^*g_*h_*\omega_X^{\otimes m}\to h_*\omega_X^{\otimes m}$$
is surjective, which implies that $h_*\omega_X^{\otimes m}$ is globally generated as well. Using what we showed above
for fiber spaces that are smooth in codimension $1$, we deduce that
$$h_*\omega_X^{\otimes m}\simeq \shO_B^{\oplus P_m(H)},$$ 
where $H$ denotes the general fiber of $h$.  Furthermore, we have
$$g_*\shO_{B}\simeq\bigoplus_{\alpha\in{\rm Ker} (\hat{g})}\alpha$$
where $\hat{g}\colon {\rm Pic}^0 (A) \to {\rm Pic}^0 (B)$ is the dual isogeny of $g$, hence $f_*\omega_X^{\otimes m}$ is a direct sum of torsion line bundles on $A$. Since $f_*\omega_X^{\otimes m}$ is globally generated, we obtain the same conclusion ($\ref{eq1}$).

\noindent
\emph{Step 2.} 
We deduce  next the general case from the statement for globally generated sheaves proved in Step 1. By the same Stein factorization argument as above, we may assume that $f$ is a fiber space.

We now use the statement and notation of Theorem \ref{thm:LPS1}. By Step 1, we have that
$$\varphi^* f_* \omega_X^{\otimes m} \simeq f'_*\omega_{X'}^{\otimes m}\simeq \shO_{A'}^{\oplus P_m(F)},$$
which implies that $f_* \omega_X^{\otimes m}$ is a direct summand of $(\varphi_*\shO_{A'})^{\oplus P_m(F)}$. Since
$$\varphi_*\shO_{A'}\simeq\bigoplus_{\beta\in{\rm Ker} (\hat{\varphi})}\beta$$
where $\hat{\varphi}$ is the dual isogeny of $\varphi$, we deduce that  
$$V^0(A, f_* \omega_X^{\otimes m})\subseteq{\rm Ker} (\hat{\varphi})$$ 
and in particular $\dim V^0(A, f_* \omega_X^{\otimes m})=0$. 

We consider again the Chen-Jiang decomposition of $f_* \omega_X^{\otimes m}$ provided by 
Theorem \ref{thm:LPS2}. By \cite[Lemma 3.3]{LPS}, we have 
$$V^0(A, f_* \omega_X^{\otimes m})=\bigcup_{i\in I}\alpha_i^{-1}\otimes p_i^*\Pic^0(A_i).$$
Since $\dim V^0(A, f_* \omega_X^{\otimes m})=0$, we deduce that $\dim A_i=0$ for all $i\in I$, which in turn leads to a 
decomposition
$$f_* \omega_X^{\otimes m} \simeq \bigoplus_{i =1}^{P_m (F)} \alpha_i$$
with $\alpha_i$ torsion line bundles on $A$, as desired.

\subsection{Proof of Theorem \ref{main}}
By possibly passing to a further log resolution, we may assume from the beginning that $E : = X \smallsetminus U$ is a simple normal crossing divisor, and therefore
$$\kappa (U ) = \kappa (X, K_X + E).$$

We start with part (2), which is mostly a quick application of Theorem \ref{trivial}. Indeed, we consider the \'etale base change $X' \to X$ as in Theorem \ref{thm:LPS1}, so that $f'_* \omega_{X'}^{\otimes m}$ is globally generated for $m\ge 1$. According to Theorem \ref{trivial}, it follows that 
$$f'_* \omega_{X'}^{\otimes m} \simeq \shO_{A'}^{\oplus P_m (F)},$$
which gives $P_m (X') = P_m (F)$, and in particular $\kappa (X) = \kappa (X') = \kappa (F)$.
We also obviously have $P_m (X') \ge P_m (X)$.

We are left with proving that $\kappa (X) = \kappa (U)$. Given the interpretation of $\kappa (U)$ above, it suffices to 
show that 
\begin{equation}\label{eqn:exc}
f_* \omega_X^{\otimes m}\simeq f_* (\omega_X^{\otimes m}(mE)) \,\,\,\,\,\,{\rm for ~all} \,\,\,\,m\ge 1.
\end{equation}
We consider the diagram
\[
\begin{tikzcd}
U \dar{f_U} \rar{i} & X \dar{f} \\
V \rar{j} & A
\end{tikzcd}
\]
where the horizontal maps are the inclusions, and $f_U$ is the restriction of $f$ to $U$.
By Theorem \ref{trivial} we know that $f_* \omega_X^{\otimes m}$ is locally free for every $m\ge 1$, and therefore the natural morphism
$$f_* \omega_X^{\otimes m}\to j_*j^*f_* \omega_X^{\otimes m}$$
is an isomorphism, since $\codim_A (A \smallsetminus V)\ge 2$. By base change we deduce that
$$f_* \omega_X^{\otimes m}\simeq  j_*{f_U}_*i^*\omega_X^{\otimes m} \simeq  f_*i_*i^*\omega_X^{\otimes m}\simeq  \bigcup_{k\ge 0}f_*(\omega_X^{\otimes m}(kE)),$$
which immediately implies ($\ref{eqn:exc}$). (Note that 
$$i_*i^*\omega_X^{\otimes m} \simeq \omega_X^{\otimes m} \otimes_{\shO_X} \shO_X(*E) = \bigcup_{k\ge 0} \omega_X^{\otimes m}(kE),$$ 
where $ \shO_X(*E)$ is the quasi-coherent sheaf of functions with poles of any order along $E$.)

\medskip

The rest of the section is devoted to proving (1), essentially by reducing it to (2). 
We consider the closed subset $Z:=A\smallsetminus V$, and decompose it as $Z = D\cup W$, where each irreducible component of $D$ has codimension $1$, while $W$ has codimension at least $2$. By Lemma \ref{lkd} we have
$$\kappa(V) = \kappa(A, D).$$
We conclude that in order to finish the proof, we need to show that
\begin{equation}\label{addition-LGT}
\kappa(A, D) + \kappa (F) =  \kappa (U) = \kappa (X, K_X + E).
\end{equation}
If $D = \emptyset$, then we are done by (2), so we assume that $D$ is non-trivial. There are two cases, according to whether $D$ is an ample divisor or not.

If $D$ is ample, then $\kappa (A, D)  = \dim A$, and the formula in ($\ref{addition-LGT}$) follows from the additivity of 
the log Kodaira dimension for fiber spaces over varieties of log general type \cite[Corollary 2]{Maehara} (see also the 
earlier \cite[Theorem 30]{Kawamata} for the case $\kappa (U) \ge 0$).

If $D$ is not ample, a well-known structural fact says that there exist a fibration $q\colon A \to B$ of abelian varieties and an ample effective divisor $H$ on $B$ such that $\dim A > \dim B$ and $D=q^*H$.
\[
\begin{tikzcd}
X \rar{f} \arrow[bend right=40]{rr}{g} & A \rar{q} & B
\end{tikzcd}
\]
For a point $b \in B$, we denote the fiber of $g$ over $b$ by $X_b$, and the fiber of $q$ over $b$ by $A_b$. Consider the following open algebraic fiber space\footnote{This means that $t$ is a dominant morphism of smooth quasi-projective varieties, with irreducible general fiber.}
\[
\begin{tikzcd}
U=X\smallsetminus E \rar \arrow[bend right=40]{rr}{t} & A\smallsetminus D \rar & B \smallsetminus H
\end{tikzcd}
\]
By Lemma \ref{lkd}, we have 
$$\kappa(B \smallsetminus H)=\kappa (B, H) = \dim B,$$ 
and since $q$ is a fiber space, we also have 
$$\kappa (V) = \kappa (A, D)=\kappa (B, H).$$  
On the other hand, \cite[Corollary 2]{Maehara} used above continues to apply in this more general setting, to give 
$$\kappa (B \smallsetminus H) + \kappa (U_b) = \kappa (U),$$
where $b \in B \smallsetminus H$. (Note that this result does not require the smoothness of $t$.) We can choose $b$ sufficiently general so that $\codim_{A_b} (A_b\cap W) \ge 2$. Since $A_b \cap D = \emptyset$, we then also have 
$$\codim_{A_b}  (A_b \cap Z)\ge 2.$$
Moreover, $b$ can be chosen sufficiently general so that $E_b:=E|_{X_b}$ has simple normal crossings, and thus 
$\kappa(U_b)=\kappa(X_b, K_{X_b}+E_b)$.
We obtain a fiber space $f_b\colon X_b\to A_b$ with general fiber $F$,  induced by the restriction of $f$ to $X_b$, which is smooth over the complement of the closed subset $A_b \cap Z$ of codimension at least $2$. We have 
$$f_b^{-1}(A_b \cap Z)=f^{-1}(Z)\cap X_b=E\cap X_b=E_b.$$
By part (2) we deduce that $\kappa (U_b)=\kappa(X_b\smallsetminus E_b)=\kappa(F)$, and thus
finally 
$$\kappa (V) + \kappa (F)=\kappa (B \smallsetminus H)+\kappa (U_b)=\kappa (U).$$

\subsection{Proof of Corollary \ref{splitting}}
Under the assumption that $f \colon X \to A$ is smooth away from a closed set of codimension at least $2$, by Theorem \ref{trivial} and 
Theorem \ref{thm:LPS1} we may assume (after having passed to an \'etale base change $A' \to A$, and noting that the relative canonical model is compatible with this base change) that we have
$$f_* \omega_X^{\otimes m} \simeq \shO_A^{\oplus P_m(F)}$$
for all $m\ge 0$ (hence $P_m (X) = P_m (F)$). We deduce that the evaluation morphisms
$$H^0(X, \omega_X^{\otimes m})\otimes_\mathbb{C}\shO_A\simeq H^0(A, f_* \omega_X^{\otimes m})\otimes_\mathbb{C}\shO_A\to f_* \omega_X^{\otimes m}$$
are all isomorphisms. Putting these together we obtain an $\shO_A$-algebra isomorphism
$$\big(\bigoplus_{m\ge0} H^0(X, \omega_X^{\otimes m})\big)\otimes_\mathbb{C}\shO_A\simeq \bigoplus_{m\ge0}f_* \omega_X^{\otimes m},$$
 since the evaluation morphisms for each $f_* \omega_X^{\otimes m}$ are compatible with the multiplicative structures on the two 
sides. Recalling that $X_{{\rm can}, f}=  \textbf{Proj} \big(\bigoplus_{m\ge0}f_* \omega_X^{\otimes m}\big)$, we obtain
$$X_{{\rm can}, f} \simeq X_{\rm can} \times A.$$
Moreover, by the invariance of plurigenera and base change we also have a natural $\mathbb{C}$-algebra isomorphism
$$\bigoplus_{m\ge0}H^0(X, \omega_X^{\otimes m})\simeq \bigoplus_{m\ge0}H^0(F, \omega_F^{\otimes m}),$$
which shows that $X_{\rm can} \simeq F_{\rm can}$.

\subsection{Proof of Theorem \ref{mad}}
We start with a general set-up.  Since $Y$ is of maximal Albanese dimension, we can take a Stein factorization of the Albanese morphism $a_Y\colon Y\to\Alb(Y)$ such that $a_Y=g\circ h$, where $h\colon Y\to Z$ is birational, $g\colon Z\to \Alb(Y)$ is finite onto its image, and $Z$ is a normal projective variety. 

By \cite[Theorem 13]{Kawamata}, there exists an \'etale cover $\varphi\colon Z'\to Z$, an abelian variety $A$, and a normal projective variety $W$ such that 
$$Z'\simeq W\times A \,\,\,\,\,\,{\rm and} \,\,\,\,\,\, \kappa(W)=\dim W=\kappa(Z) = \kappa (Y).$$
 Denote the projection of $Z'$ onto $W$ by $p$. We consider the following commutative diagram, where the morphisms in the middle column are obtained by base change from $h$ and $f$ via $\varphi$ and then $\varphi'$, while on the left we have the respective fibers over a general point $w \in W$. (The horizontal maps between the two left columns are all inclusions.)
\begin{equation}\label{eq2}
\begin{tikzcd}
X'_w\rar \dar{f'_w} & X' \dar{f'} \rar{\varphi''} & X \dar{f}  \\
Y'_w\rar{i_w} \dar{h'_w} & Y' \dar{h'} \rar{\varphi'} & Y \dar{h}  \\
 A= A_w \dar \rar{} & Z' \dar{p} \rar{\varphi} & Z\\
\{w\}\rar& W
\end{tikzcd}
\end{equation}
The projective varieties $X'$ and $Y'$ are both smooth, since $\varphi$ is \'etale. We can choose $w$ sufficiently general so that $X'_w$ and $Y'_w$ are smooth, and $h'_w$ is birational. Moreover $f'_w$ is a fiber space.

After this preparation, we are ready to prove Theorem \ref{mad}. By possibly passing to a further log resolution, we may assume that $E : = X \smallsetminus U$ is a simple normal crossing divisor, and therefore
$$\kappa (U ) = \kappa (X, K_X + E).$$

We start again with part (2), which is somewhat less involved. If $T = Y \smallsetminus V$, recall that in 
(2) we are assuming that the codimension of $T$ is at least $2$.  Note that $f^{-1} (T) = E$. 
The morphism $f'$ is smooth over $Y'\smallsetminus \varphi'^{-1}(T)$, and $\codim_{Y'} \varphi'^{-1}(T)\ge 2$.
Since $\varphi''$ is \'etale, we deduce that $E':=\varphi''^*E$ has simple normal crossings, hence the pair 
$(X', E')$ is log canonical. We have 
$$\kappa(X', K_{X'}+E')=\kappa(X', \varphi''^*(K_X+E))=\kappa(X, K_X+E)=\kappa(U).$$
If $T'_w:=(\varphi'\circ i_w)^{-1}(T)$, we also know that $f'_w$ is smooth over $Y'_w\smallsetminus T'_w$ as well, while choosing 
$w$ carefully we can also ensure that  $\codim_{Y'_w} T'_w \ge 2$. We can choose $w$ sufficiently general such that
$E'_w:=E'|_{X'_w}$ has simple normal crossings, and thus 
$$\kappa(X'_w\smallsetminus E'_w)=\kappa(X'_w, K_{X'_w}+E'_w).$$
As in the proof of Theorem \ref{main}, since $W$ is of general type, after passing to a resolution $\widetilde{W}$ of $W$ and a log resolution of the pull-back of the pair $(X', E')$ to the main component of the fiber product $X' \times_W \widetilde{W}$, by \cite[Corollary 2]{Maehara}  we deduce the additivity of log Kodaira dimension 
$$\kappa(X', K_{X'}+E')=\kappa(X'_w, K_{X'_w}+E'_w)+\dim W.$$
Since $h'_w$ is a birational morphism, it is immediate to check that  $h'_w\circ f'_w$ is also smooth over the complement in $A$ of the closed subset $S_w=h'_w(T'_w\cup{\rm Exc}(h'_w))$ of codimension at least $2$, and its general fiber is $F$.

At this stage, by Theorem \ref{main}(2) we have
$$\kappa\big(X'_w\smallsetminus (h'_w\circ f'_w)^{-1}(S_w)\big)=\kappa(X'_w)=\kappa(F).$$
Since $E'_w=(f'_w)^{-1}(T'_w)\subseteq (h'_w\circ f'_w)^{-1}(S_w)$, we also have
$$\kappa\big(X'_w\smallsetminus (h'_w\circ f'_w)^{-1}(S_w)\big)\ge\kappa(X'_w\smallsetminus E'_w)\ge\kappa(X'_w),$$
and thus $\kappa(X'_w\smallsetminus E'_w)=\kappa(X'_w)=\kappa(F)$. Putting the identities above together, we deduce that
$$\kappa(U)=\kappa(X', K_{X'}+E')=\kappa(X'_w, K_{X'_w}+E'_w)+\dim W=\kappa(F)+\kappa(Y).$$
On the other hand, by the subadditivity of the Kodaira dimension \cite[Theorem 1.1]{HPS}, we have $\kappa(X)\ge\kappa(F)+\kappa(Y)$. Thus
$$\kappa(F)+\kappa(Y)=\kappa(U)\ge\kappa(X)\ge\kappa(F)+\kappa(Y).$$
Note also that since $T$ has codimension at least $2$ in $Y$, by Lemma \ref{lkd} we have $\kappa(Y)=\kappa(V)$. We 
finally obtain 
$$\kappa (X) = \kappa (U)  = \kappa (V) + \kappa (F) = \kappa (Y) + \kappa (F).$$

\medskip

We now prove (1). We note first that we can assume $Y \smallsetminus V$ to be a simple
normal crossing divisor. Let $\mu\colon \widetilde{Y}\to Y$ be a log resolution of $Y \smallsetminus V$ which is an isomorphism over $V$. In particular $\mu^{-1}(Y\smallsetminus V)$ is a divisor with simple normal crossing support, and $\widetilde{Y}$ is still of maximal Albanese dimension. We consider the commutative diagram
\[
\begin{tikzcd}
\widetilde{X}\rar \drar{\widetilde{f}} & (X\times_Y \widetilde{Y})_{\rm{main}} \dar{m} \rar{} & X \dar{f}  \\
 & \widetilde{Y}  \rar{\mu} & Y 
\end{tikzcd}
\]
where $(X\times_Y \widetilde{Y})_{\rm{main}}$ is the main component of $X\times_Y \widetilde{Y}$, and  $\widetilde{X}$ is a resolution which is an isomorphism over $(\mu\circ m)^{-1}(V)$.  The algebraic fiber space $\widetilde{f}\colon \widetilde{X}\to \widetilde{Y}$ is smooth over $\mu^{-1}(V) \simeq V$, with general fiber $F$. Thus we can assume from the start that $D:=Y\smallsetminus V$ is a simple normal crossing divisor, and 
$$\kappa (V)=\kappa(Y, K_Y+D).$$

We next show that we can also assume that there exists a birational morphism $h\colon Y\to A$, where $A$ is an abelian variety. We consider again the diagram ($\ref{eq2}$) in our set-up, but we now keep track of the divisor $D$ and its pullbacks as well.  Since $\varphi'$ is \'etale, $D':=\varphi'^*D$ is also a simple normal crossing divisor. 
For $w$ sufficiently general, in addition to all the properties described in the set-up and the proof of (1), we can also assume that 
$D'_w:=D'|_{Y'_w}$ is a simple normal crossing divisor. Since $W$ is of general type, as before we have
 the additivity of log Kodaira dimension 
$$\kappa(X, K_X+E)=\kappa(X', K_{X'}+E')=\kappa(X'_w, K_{X'_w}+E'_w)+\dim W,$$
and similarly
$$\kappa(Y, K_Y+D)=\kappa(Y', K_{Y'}+D')=\kappa(Y'_w, K_{Y'_w}+D'_w)+\dim W.$$
Now the general fiber of $f'_w$ is $F$. Thus to obtain the conclusion $\kappa (V) + \kappa (F) \ge \kappa (U)$, we only need to prove the inequality
$$\kappa(Y'_w, K_{Y'_w}+D'_w)+ \kappa (F)\ge \kappa(X'_w, K_{X'_w} + E'_w).$$
Note that the morphism $f'_w$ is smooth over $V'_w:=(\varphi'\circ i_w)^{-1}(V) = Y'_w\smallsetminus D'_w$, and 
$$\kappa(V'_w)=\kappa(Y'_w, K_{Y'_w}+D'_w),$$
while $U'_w:={f'_w}^{-1}(V'_w) = X'_w\smallsetminus E'_w$, and
$$\kappa(U'_w)=\kappa(X'_w, K_{X'_w}+E'_w).$$
This allows us indeed to assume from the start (replacing $f\colon X \to Y$ by  $f'_w\colon X'_w \to Y'_w$) that there exists a birational morphism $h\colon Y\to A$, where $A$ is an abelian variety.

The picture is now
\[
\begin{tikzcd}
X \rar{f} \arrow[bend right=40]{rr}{g} & Y \rar{h} & A
\end{tikzcd}
\]
At this stage we can forget about the previous steps in the proof, and reuse some of the notation symbols. 
The birational morphism $h$ is an isomorphism away from $Z:=h\big(\Supp (D)\cup{\rm Exc}(h)\big)$ 
and thus $g$ is smooth over $A\smallsetminus Z$. By Theorem \ref{main}(1), we then have
$$\kappa (A\smallsetminus Z) + \kappa (F) = \kappa \big(X \smallsetminus g^{-1}(Z)\big).$$
We can take a log resolution $\mu\colon Y'\to Y$ which is an isomorphism over $Y\smallsetminus \big(\Supp (D)\cup{\rm Exc} (h)\big)$, such that the support of $\mu^{-1}\big(\Supp (D)\cup{\rm Exc} (h)\big)\cup{\rm Exc}(\mu)$ is a simple normal crossing divisor $D'$.  Since ${\rm Supp} \big(h^{-1}(Z)\big)=\Supp (D) \cup{\rm Exc} (h)$,  we deduce that
$$\kappa (A\smallsetminus Z)=\kappa(Y', K_{Y'}+D').$$
Now $K_{Y'}\sim\mu^*K_Y+E'$, where $E'$ is an effective divisor whose support is ${\rm Exc}(\mu)$. Since $A$ is an abelian variety, we also have $K_Y\sim E$, where $E$ is an effective and $h$-exceptional divisor supported on ${\rm Exc} (h)$. We deduce that 
$$K_{Y'}+D'\sim\mu^*E+E'+D'\ge 0,$$ 
and since $\Supp(\mu^*E+E'+D')=\Supp(\mu^*E+E'+\mu^*D)$, we finally have
$$\kappa (A\smallsetminus Z)=\kappa(Y', K_{Y'}+D')=\kappa(Y', \mu^*E+E'+D')$$
$$=\kappa(Y', \mu^*E+E'+\mu^*D) = \kappa(Y', \mu^*(K_Y +D) + E') =\kappa(Y, K_Y+D)=\kappa (V).$$
Putting everything together, we obtain
$$\kappa (V) + \kappa (F) = \kappa (X\smallsetminus g^{-1}(Z)).$$
Finally, since $g^{-1}(Z)=f^{-1}(h^{-1}(Z))=f^{-1}(\Supp (D) \cup{\rm Exc} (h))$ contains $f^{-1}(D)$, which is supported on $E$,  we deduce that 
$$\kappa (V) + \kappa (F) \ge \kappa (X\smallsetminus E)=\kappa(U).$$

\subsection{Proof of Theorem \ref{subadditivity}}
It is clear that 
$$\kappa (X, K_X+ E)\ge \kappa (X, K_X+(f^*D)_{\rm red}),$$
hence to prove the theorem it suffices to assume $E = (f^*D)_{\rm red}$.

We first show that we can reduce to the case when $E$ has simple normal crossings. Take a log resolution $\mu\colon Y\to X$ of the pair $(X, E)$, so that $\mu^* E $ has simple normal crossing support, and 
let $K_Y=\mu^*K_X+G$ with $G$ an effective $\mu$-exceptional divisor on $Y$. We may assume that $\mu$ is an isomorphism over $X\smallsetminus E$. Thus the general fiber of $g = f\circ \mu$ is isomorphic to $F$.
\[
\begin{tikzcd}
Y \rar{\mu} \arrow[bend right=40]{rr}{g} & X \rar{f} & A
\end{tikzcd}
\]
We have that $\mu^* E \ge (g^*D)_{\rm red}$, and since $G$ is effective and $\mu$-exceptional, we deduce that
$$\kappa (X, K_X+ E)=\kappa (Y, \mu^*(K_X+ E)+G)$$
$$=\kappa (Y, K_Y+\mu^*E)\geq \kappa (Y, K_Y+(g^*D)_{\rm red})\ge  \kappa (F) + \kappa (A, D),$$
where the last inequality holds if we can prove the result in the simple normal crossing case.

We proceed therefore assuming that $E$ has simple normal crossings. Let $k \ge 1$ be the maximal coefficient of the irreducible components of $f^*D$. Thus the coefficients of the divisor $E-\frac{1}{k}f^*D$ are nonnegative and strictly smaller than $1$. Since it has simple normal crossing support, we deduce that the log pair $(X, E-\frac{1}{k}f^*D)$ is klt. We consider the divisor class
$$T:=kK_X+k E-f^*D\sim_{\mathbb{Q}} k(K_X+ E-\frac{1}{k}f^*D).$$
By \cite[Theorem 1.1]{Meng}, generalizing Theorem \ref{thm:LPS1}, there exists an isogeny $\varphi \colon A' \to A$  and 
an induced fiber product
\[
\begin{tikzcd}
X' \dar{f'} \rar{\varphi'} & X \dar{f} \\
A' \rar{\varphi} & A
\end{tikzcd}
\]
such that $ f'_* \shO_{X'}(\varphi'^*(mT))\simeq \varphi^*f_*\shO_X(mT)$ is globally generated for all $m\ge 1$. 
Note that by the base change theorem
$${\rm rk} \big( f_*\shO_X(mT) \big)=P_{mk}(F).$$
We conclude that for each $m\ge 1$ there exist injective coherent sheaf homomorphisms
$$\shO_{A'}^{\oplus P_{mk}(F)}\hookrightarrow f'_* \shO_{X'}(\varphi'^*(mT)).$$
Tensoring with $\shO_{A'}(\varphi^*(mD))$, and noting that 
$$f'_* \shO_{X'}(\varphi'^*(mT))\otimes\shO_{A'}(\varphi^*(mD))\simeq 
f'_*\shO_{X'}\big(\varphi'^*(mk(K_X+E))\big),$$
this leads to injections
$$\shO_{A'}(\varphi^*(mD))^{\oplus P_{mk}(F)}\hookrightarrow f'_* \shO_{X'}\big(mk (K_{X'}+ \varphi'^*E)\big).$$
Thus we deduce that
$$P_{mk} (X',  \varphi'^*E) \ge P_{mk} (F) \cdot P_m (A', \varphi^*D) \ge P_{mk} (F) \cdot P_m (A, D)\footnote{This is analogous to Corollary \ref{old}.}$$
for all $m \ge 0$. Since $k$ is fixed and $\varphi'^*\omega_X \simeq \omega_{X'}$, it is immediate to see that this implies 
$$\kappa(X', K_{X'} +\varphi'^*E)\ge \kappa(F)+\kappa(A, D).$$
Now $\varphi'$ is \'etale, hence $\kappa(X', K_{X'} +\varphi'^*E)=\kappa(X, K_X+ E)$, and so 
finally 
$$\kappa(X, K_X+ E) \ge \kappa(F)+\kappa(A, D).$$

\noindent
{\bf Acknowledgements.} The first author would like to thank the Department of Mathematics at Harvard University for
its hospitality during the preparation of this paper.

\end{document}

%% file: macros.tex
\usepackage{amsmath,amsfonts,amsthm,mathrsfs}
\usepackage{amssymb}

\usepackage[unicode,bookmarks]{hyperref}
\usepackage[usenames,dvipsnames]{xcolor}
\hypersetup{colorlinks=true,citecolor=NavyBlue,linkcolor=Brown,urlcolor=Orange}

\usepackage[alphabetic,initials]{amsrefs}

\usepackage{enumitem}

\usepackage{chngcntr}

\ifdraft
\usepackage[notcite,notref,color]{showkeys}

\definecolor{labelkey}{gray}{0.5}
\fi

\usepackage{tikz}

\usepackage{tikz-cd}

\usetikzlibrary{matrix,arrows}
\newlength{\myarrowsize} 

\pgfarrowsdeclare{cmto}{cmto}{
	\pgfsetdash{}{0pt} 
	\pgfsetbeveljoin 
	\pgfsetroundcap 
	\setlength{\myarrowsize}{0.6pt}
	\addtolength{\myarrowsize}{.5\pgflinewidth}
	\pgfarrowsleftextend{-4\myarrowsize-.5\pgflinewidth} 
	\pgfarrowsrightextend{.8\pgflinewidth}
}{
	\setlength{\myarrowsize}{0.6pt} 
  	\addtolength{\myarrowsize}{.5\pgflinewidth}  
	\pgfsetlinewidth{0.5\pgflinewidth}
	\pgfsetroundjoin
	\pgfpathmoveto{\pgfpoint{1.5\pgflinewidth}{0}}
	\pgfpatharc{-109}{-170}{4\myarrowsize}
	\pgfpatharc{10}{189}{0.58\pgflinewidth and 0.2\pgflinewidth}
	\pgfpatharc{-170}{-115}{4\myarrowsize+\pgflinewidth}
	\pgfpathclose
	\pgfusepathqfillstroke
	\pgfpathmoveto{\pgfpoint{1.5\pgflinewidth}{0}}
	\pgfpatharc{109}{170}{4\myarrowsize}
	\pgfpatharc{-10}{-189}{0.58\pgflinewidth and 0.2\pgflinewidth}
	\pgfpatharc{170}{115}{4\myarrowsize+\pgflinewidth}
	\pgfpathclose
	\pgfusepathqfillstroke
	\pgfsetlinewidth{2\pgflinewidth}
}

\pgfarrowsdeclare{cmonto}{cmonto}{
	\pgfsetdash{}{0pt} 
	\pgfsetbeveljoin 
	\pgfsetroundcap 
	\setlength{\myarrowsize}{0.6pt}
	\addtolength{\myarrowsize}{.5\pgflinewidth}
	\pgfarrowsleftextend{-4\myarrowsize-.5\pgflinewidth} 
	\pgfarrowsrightextend{.8\pgflinewidth}
}{
	\setlength{\myarrowsize}{0.6pt} 
  	\addtolength{\myarrowsize}{.5\pgflinewidth}  
	\pgfsetlinewidth{0.5\pgflinewidth}
	\pgfsetroundjoin
	\pgfpathmoveto{\pgfpoint{1.5\pgflinewidth}{0}}
	\pgfpatharc{-109}{-170}{4\myarrowsize}
	\pgfpatharc{10}{189}{0.58\pgflinewidth and 0.2\pgflinewidth}
	\pgfpatharc{-170}{-115}{4\myarrowsize+\pgflinewidth}
	\pgfpathclose
	\pgfusepathqfillstroke
	\pgfpathmoveto{\pgfpoint{1.5\pgflinewidth}{0}}
	\pgfpatharc{109}{170}{4\myarrowsize}
	\pgfpatharc{-10}{-189}{0.58\pgflinewidth and 0.2\pgflinewidth}
	\pgfpatharc{170}{115}{4\myarrowsize+\pgflinewidth}
	\pgfpathclose
	\pgfusepathqfillstroke
	\pgfpathmoveto{\pgfpoint{1.5\pgflinewidth-0.3em}{0}}
	\pgfpatharc{-109}{-170}{4\myarrowsize}
	\pgfpatharc{10}{189}{0.58\pgflinewidth and 0.2\pgflinewidth}
	\pgfpatharc{-170}{-115}{4\myarrowsize+\pgflinewidth}
	\pgfpathclose
	\pgfusepathqfillstroke
	\pgfpathmoveto{\pgfpoint{1.5\pgflinewidth-0.3em}{0}}
	\pgfpatharc{109}{170}{4\myarrowsize}
	\pgfpatharc{-10}{-189}{0.58\pgflinewidth and 0.2\pgflinewidth}
	\pgfpatharc{170}{115}{4\myarrowsize+\pgflinewidth}
	\pgfpathclose
	\pgfusepathqfillstroke
	\pgfsetlinewidth{2\pgflinewidth}
}

\pgfarrowsdeclare{cmhook}{cmhook}{
	\pgfsetdash{}{0pt} 
	\pgfsetbeveljoin 
	\pgfsetroundcap 
	\setlength{\myarrowsize}{0.6pt}
	\addtolength{\myarrowsize}{.5\pgflinewidth}
	\pgfarrowsleftextend{-4\myarrowsize-.5\pgflinewidth} 
	\pgfarrowsrightextend{.8\pgflinewidth}
}{
	\setlength{\myarrowsize}{0.6pt} 
  	\addtolength{\myarrowsize}{.5\pgflinewidth}  
 	\pgfsetdash{}{0pt}
	\pgfsetroundcap
	\pgfpathmoveto{\pgfqpoint{0pt}{-4.667\pgflinewidth}}
	\pgfpathcurveto
    {\pgfqpoint{4\pgflinewidth}{-4.667\pgflinewidth}}
    {\pgfqpoint{4\pgflinewidth}{0pt}}
    {\pgfpointorigin}
	\pgfusepathqstroke
}

\newenvironment{diagram*}[2]{%
\[%
\begin{tikzpicture}[>=cmto,baseline=(current bounding box.center),%
	to/.style={->,font=\scriptsize,cap=round},%
	into/.style={cmhook->,font=\scriptsize,cap=round},%
	onto/.style={-cmonto,font=\scriptsize,cap=round},%
	math/.style={matrix of math nodes, row sep=#2, column sep=#1,%
		text height=1.5ex, text depth=0.25ex}]%
}{%
\end{tikzpicture}%
\]%
\ignorespacesafterend%
}

\newcommand{\derR}{\mathbf{R}}

\newcommand{\shH}{\mathcal{H}}

\newcommand{\tensor}{\otimes}

\DeclareMathOperator{\Supp}{Supp}
\DeclareMathOperator{\codim}{codim}

\DeclareMathOperator{\Alb}{Alb}
\DeclareMathOperator{\Pic}{Pic}

\newcommand*{\sheafhom}{\mathscr{H}\kern -.5pt om}

\newcommand{\shf}[1]{\mathscr{#1}}

\def\overbar#1#2#3{{%
	\setbox0=\hbox{\displaystyle{#1}}%
	\dimen0=\wd0
	\advance\dimen0 by -#2 
	\vbox {\nointerlineskip \moveright #3 \vbox{\hrule height 0.3pt width \dimen0}%
		\nointerlineskip \vskip 1.5pt \box0}%
}}

\newcommand{\pu}{p^{\ast}}

\newcommand{\shF}{\shf{F}}
\newcommand{\shG}{\shf{G}}

\newcommand{\shO}{\shf{O}}

\makeatletter
\let\@@seccntformat\@seccntformat
\renewcommand*{\@seccntformat}[1]{%
  \expandafter\ifx\csname @seccntformat@#1\endcsname\relax
    \expandafter\@@seccntformat
  \else
    \expandafter
      \csname @seccntformat@#1\expandafter\endcsname
  \fi
    {#1}%
}
\newcommand*{\@seccntformat@subsection}[1]{%
  \textbf{\csname the#1\endcsname.}
}
\makeatother

\makeatletter
\let\@paragraph\paragraph
\renewcommand*{\paragraph}[1]{%
	\vspace{0.3\baselineskip}%
	\@paragraph{\textit{#1}}%
}
\makeatother

\counterwithin{equation}{subsection}
\counterwithout{subsection}{section}
\counterwithin{figure}{subsection}

\newtheorem{theorem}[equation]{Theorem}
\newtheorem*{theorem*}{Theorem}
\newtheorem{lemma}[equation]{Lemma}
\newtheorem*{lemma*}{Lemma}
\newtheorem{corollary}[equation]{Corollary}

\newtheorem*{proposition*}{Proposition}

\theoremstyle{definition}

\newtheorem*{definition*}{Definition}
\newtheorem{remark}[equation]{Remark}

\newtheorem*{example*}{Example}
\newtheorem*{problem*}{Problem}

\theoremstyle{plain}

\newcommand{\theoremref}[1]{\hyperref[#1]{Theorem~\ref*{#1}}}
\newcommand{\lemmaref}[1]{\hyperref[#1]{Lemma~\ref*{#1}}}
\newcommand{\definitionref}[1]{\hyperref[#1]{Definition~\ref*{#1}}}
\newcommand{\propositionref}[1]{\hyperref[#1]{Proposition~\ref*{#1}}}
\newcommand{\conjectureref}[1]{\hyperref[#1]{Conjecture~\ref*{#1}}}
\newcommand{\corollaryref}[1]{\hyperref[#1]{Corollary~\ref*{#1}}}
\newcommand{\exampleref}[1]{\hyperref[#1]{Example~\ref*{#1}}}

\makeatletter
\let\old@caption\caption
\renewcommand*{\caption}[1]{%
	\setcounter{figure}{\value{equation}}%
	\stepcounter{equation}%
	\old@caption{#1}\relax%
}
\makeatother

\newcounter{intro}

\newtheorem{intro-conjecture}[intro]{Conjecture}
\newtheorem{intro-corollary}[intro]{Corollary}
\newtheorem{intro-theorem}[intro]{Theorem}

\newcommand{\parref}[1]{\hyperref[#1]{\S\ref*{#1}}}

\makeatletter
\newcommand*\if@single[3]{%
  \setbox0\hbox{${\mathaccent"0362{#1}}^H$}%
  \setbox2\hbox{${\mathaccent"0362{\kern0pt#1}}^H$}%
  \ifdim\ht0=\ht2 #3\else #2\fi
  }
\newcommand*\rel@kern[1]{\kern#1\dimexpr\macc@kerna}
\newcommand*\widebar[1]{\@ifnextchar^{{\wide@bar{#1}{0}}}{\wide@bar{#1}{1}}}
\newcommand*\wide@bar[2]{\if@single{#1}{\wide@bar@{#1}{#2}{1}}{\wide@bar@{#1}{#2}{2}}}
\newcommand*\wide@bar@[3]{%
  \begingroup
  \def\mathaccent##1##2{%
    \if#32 \let\macc@nucleus\first@char \fi
    \setbox\z@\hbox{$\macc@style{\macc@nucleus}_{}$}%
    \setbox\tw@\hbox{$\macc@style{\macc@nucleus}{}_{}$}%
    \dimen@\wd\tw@
    \advance\dimen@-\wd\z@
    \divide\dimen@ 3
    \@tempdima\wd\tw@
    \advance\@tempdima-\scriptspace
    \divide\@tempdima 10
    \advance\dimen@-\@tempdima
    \ifdim\dimen@>\z@ \dimen@0pt\fi
    \rel@kern{0.6}\kern-\dimen@
    \if#31
      \overline{\rel@kern{-0.6}\kern\dimen@\macc@nucleus\rel@kern{0.4}\kern\dimen@}%
      \advance\dimen@0.4\dimexpr\macc@kerna
      \let\final@kern#2%
      \ifdim\dimen@<\z@ \let\final@kern1\fi
      \if\final@kern1 \kern-\dimen@\fi
    \else
      \overline{\rel@kern{-0.6}\kern\dimen@#1}%
    \fi
  }%
  \macc@depth\@ne
  \let\math@bgroup\@empty \let\math@egroup\macc@set@skewchar
  \mathsurround\z@ \frozen@everymath{\mathgroup\macc@group\relax}%
  \macc@set@skewchar\relax
  \let\mathaccentV\macc@nested@a
  \if#31
    \macc@nested@a\relax111{#1}%
  \else
    \def\gobble@till@marker##1\endmarker{}%
    \futurelet\first@char\gobble@till@marker#1\endmarker
    \ifcat\noexpand\first@char A\else
      \def\first@char{}%
    \fi
    \macc@nested@a\relax111{\first@char}%
  \fi
  \endgroup
}
\makeatother